\documentclass[11pt]{amsart}
\usepackage{amsfonts, amsthm, amsmath, amssymb}
\usepackage[all]{xy}
\usepackage[top=1in,bottom=1in,left=1in,right=1in]{geometry}
\usepackage{tikz}
\usepackage[pdftex]{hyperref}

\newcommand{\bC}{\mathbb{C}}

\newcommand{\bM}{\mathbb{M}}

\newcommand{\bQ}{\mathbb{Q}}
\newcommand{\bR}{\mathbb{R}}

\newcommand{\bZ}{\mathbb{Z}}
\newcommand{\unit}{\mathbf{1}}

\newcommand{\cM}{\mathcal{M}}

\newcommand{\fg}{\mathfrak{g}}

\newcommand{\fH}{\mathfrak{H}}
\newcommand{\fm}{\mathfrak{m}}

\newcommand{\fsl}{\mathfrak{sl}}

\newcommand{\ad}{\operatorname{ad}}

\newcommand{\Aut}{\operatorname{Aut}}

\newcommand{\End}{\operatorname{End}}

\newcommand{\Hom}{\operatorname{Hom}}

\newcommand{\Tr}{\operatorname{Tr}}

\newtheorem{thm}{Theorem}[section]

\newtheorem{prop}[thm]{Proposition}

\newtheorem{cor}[thm]{Corollary}

\theoremstyle{definition}
\newtheorem{defn}[thm]{Definition}

\theoremstyle{remark}
\newtheorem{rem}[thm]{Remark}

\begin{document}

\title{51 constructions of the Moonshine module}
\author{Scott Carnahan}

\begin{abstract}
We show using Borcherds products that for any fixed-point free automorphism of the Leech lattice satisfying a ``no massless states'' condition, the corresponding cyclic orbifold of the Leech lattice vertex operator algebra is isomorphic to the Monster vertex operator algebra.  This induces an ``orbifold duality'' bijection between algebraic conjugacy classes of fixed-point free automorphisms of the Leech lattice satisfying this condition and algebraic conjugacy classes of non-Fricke elements in the Monster.  We use the duality to show that non-Fricke Monstrous Lie algebras are Borcherds-Kac-Moody Lie algebras, and prove a refinement of Norton's Generalized Moonshine conjecture: the ambiguous constants relating generalized moonshine Hauptmoduln under conjugation and modular transformations are necessarily roots of unity.  We also describe a class of rank 2 Borcherds-Kac-Moody Lie algebras attached to the Conway group.
\end{abstract}

\maketitle

\tableofcontents

\subsection{Introduction}

In this paper, we extend the known cyclic orbifold constructions of the Monster vertex operator algebra $V^\natural$ to cover all fixed-point free automorphisms of the Leech lattice satisfying a ``no massless states'' condition, proving an extended version of a 1993 conjecture of Tuite.  By orbifold duality, we find that these cases exhaust all possible cyclic orbifold constructions of $V^\natural$.  The first construction along these lines was the original construction by \cite{FLM88}, using the $-1$ involution, and this construction gave birth to the notion of orbifold conformal field theory.  Some candidate constructions using prime order elements were proposed in the 1990s by \cite{M95} and a proof in the order 3 case was announced in \cite{DM94}.  However, proofs have only been fully written up relatively recently: in \cite{CLS16} for the order 3 case, and in \cite{ALY17} for the remaining primes 5, 7, and 13, after the mathematical foundations of the cyclic orbifold construction were settled in \cite{vEMS}.  We complete this particular picture by considering orbifolds for all fixed-point free automorphisms of the Leech lattice that satisfy Tuite's conditions.  Furthermore, we show that the ``anomaly-free'' condition can be removed using a more general orbifold construction.  We note that there are other, non-orbifold constructions of $V^\natural$ in the literature: \cite{GL11}, \cite{M97}.

In a pair of remarkable papers \cite{T92, T93}, Tuite argued using physical considerations that the genus zero property of McKay-Thompson series in Monstrous Moonshine follows from some conjectured properties of the Moonshine module $V^\natural$, namely:
\begin{enumerate}
\item If $g$ is anomaly-free and $T_g$ is Fricke-invariant, then the cyclic orbifold $V^\natural/g$ is isomorphic to $V^\natural$.
\item If $g$ is anomaly-free and $T_g$ is not Fricke-invariant, then the cyclic orbifold $V^\natural/g$ is isomorphic to the Leech lattice vertex operator algebra $V_\Lambda$.
\end{enumerate}
Furthermore, by analyzing twisted sectors of $V_\Lambda$, he gave strong evidence for a conjecture dual to the second assertion, that if $g$ is a fixed-point free automorphism of $\Lambda$ satisfying the ``no massless states'' condition, with an anomaly-free lift to an automorphism of $V_\Lambda$, then $V_\Lambda/g \cong V^\natural$.  Since there are 38 algebraic conjugacy classes of suitable fixed-point free automorphisms, this would give 38 constructions of $V^\natural$.

Recent progress in the theory of vertex operator algebras and orbifold constructions make it reasonable to revisit these conjectures.  In particular, \cite{vEMS} showed that the cyclic orbifold construction $V \mapsto V/g$ is well-defined for any anomaly-free finite order automorphism $g$ of a simple, $C_2$-cofinite, holomorphic vertex operator algebra of CFT type (such as $V^\natural$ and $V_\Lambda$).  Using this result, together with a method inspired by string-theoretic T-duality, the first conjecture, about Fricke automorphisms, was settled in \cite{PPV17}.  In this paper, we resolve the second conjecture, and therefore its dual.  We also use an unrolling method to define a generalization of the orbifold construction to allow automorphisms with anomalies.  This gives us 13 extra constructions of $V^\natural$, for a total of 51, and it yields a duality between fixed-point free algebraic conjugacy classes in $Co_1$ satisfying the ``no massless states'' condition and non-Fricke classes in $\bM$.  This duality was first conjectured in vague terms in Section 7 of \cite{CN79}, and further developed with explicit calculations on the level of genus zero modular functions in \cite{K85}.  Tuite then developed it into a statement about orbifold conformal field theory.

We remark that \cite{PPV16} not only gives a strong physical argument for the validity of Tuite's conjectures using string-theoretic tools, but also gives a new physical interpretation of the Atkin-Lehner involutions that yield the genus zero property in terms of T-duality.  Some of their algebraic manipulations are quite similar to ours, since they use the same Borcherds products, but where they have an elegant argument for isomorphisms $V^\natural/g \cong V_\Lambda$ based on physical reasoning, we have been unable to avoid a direct computation.

We now consider applications of the orbifold duality.  First, we show that all of the monstrous Lie algebras $\fm_g$ constructed in \cite{GM4} are Borcherds-Kac-Moody Lie algebras, using the orbifold correspondence to work out the non-Fricke case that was left over from the earlier paper.  We will not say much about the general theory of Borcherds-Kac-Moody Lie algebras, except to note that they are similar in many ways to finite dimensional simple Lie algebras: they are generated by $\fsl_2$-triples with relations encoded by a Cartan matrix, and their highest-weight representations satisfy a generalization of the Weyl character formula.  From section 4 of \cite{GM2}, if we have a homogeneous group action on such a Lie algebra, we obtain a twisted denominator formula, using the BGG-type methods of \cite{J04}.  We note that in the Fricke case, we recently showed in \cite{Fricke} that the BGG-type methods are avoidable using a decomposition introduced in \cite{J98}.  However, in the non-Fricke case, the infinitely many norm zero simple roots make such arguments difficult.

I was unable to prove the Borcherds-Kac-Moody condition in full in \cite{GM4}, because I did not know how to show that the norm zero simple roots commute in the non-Fricke case.  This commutativity turns out to follow from the commutativity of the weight 1 Lie algebra in $V_\Lambda$ through the orbifold correspondence, together with a short manipulation using a string quantization functor.  From the twisted denominator identities of $\fm_g$, we prove a refinement of the main theorem in \cite{GM4}, i.e., the resolution of Norton's Generalized Moonshine conjecture.  Norton's original statement of the conjecture \cite{N87} noted the presence of some ambiguous constants relating the Generalized Moonshine functions, and suggested that they are roots of unity.  Our original proof in \cite{GM4} only implies the constants are nonzero complex numbers, but here we obtain roots of unity whenever the functions are non-constant.

We conclude with a brief analysis of rank 2 Lie algebras attached to $g^*$-twisted modules of $V_\Lambda$ where $g^*$ acts without fixed points on $\Lambda$.  The constructions are analogous to the Monstrous Lie algebras $\fm_g$ constructed in \cite{GM4}, and we show using the orbifold correspondence that they are in fact isomorphic.  We briefly speculate on an analogue of generalized moonshine for $V_\Lambda$, at least for fixed-point free twists.

\subsection{Overview}

In section \ref{sect:defs}, we give basic definitions from the theory of vertex operator algebras and abelian intertwining algebras.
In section \ref{sect:lifts}, we give an overview of Tuite's conditions on automorphisms of the Leech lattice $\Lambda$.
In section \ref{sect:anomalous}, we generalize the cyclic orbifold construction given in \cite{vEMS} to allow for anomalous conformal weights.
In section \ref{sect:non-fricke}, we show that the cyclic orbifold of $V^\natural$ by a non-Fricke element is isomorphic to $V_\Lambda$, with the dual automorphism $g^*$ fixed-point free, and hence, the corresponding cyclic orbifold of $V_\Lambda$ by $g^*$ is isomorphic to $V^\natural$.
In section \ref{sect:monstrous}, we show that the monstrous Lie algebras $\fm_g$ constructed in \cite{GM4} are Borcherds-Kac-Moody for all $g \in \bM$.  The Fricke case was done in \cite{GM4}, so we resolve the non-Fricke case here.  We use this to show that all generalized moonshine functions are Hecke-monic.
In section \ref{sect:generalized}, we use the Hecke-monic property to prove a refinement of our resolution of Norton's Generalized Moonshine conjecture.
In section \ref{sect:conway}, we analyze a class of Lie algebras that are constructed in a manner similar to the monstrous Lie algebras, using abelian intertwining algebras attached to fixed-point free automorphisms of $\Lambda$ satisfying the ``no massless states'' condition.
In section \ref{sect:appendix}, we produce a table of non-anomalous non-Fricke elements, the non-negative eta-product expansions of their McKay-Thompson series, norm zero root multiplicities of the corresponding Lie algebras, and the corresponding elements of $Co_0$ in GAP and ATLAS notations.

\subsection{Acknowledgments}

The author would like to thank Toshiyuki Abe, Terry Gannon, and Michael Tuite for helpful discussions, and Ching-Hung Lam, Sven M\"oller, and Hiroki Shimakura for pointing out errors in an earlier version of this article.  This research was partly funded by JSPS Kakenhi Grant-in-Aid for Young Scientists (B) 17K14152.

\section{Vertex algebras and abelian intertwining algebras} \label{sect:defs}

We introduce some basic terminology for vertex algebras \cite{B86}, vertex operator algebras \cite{FLM88}, and abelian intertwining algebras \cite{DL93}.  We will not explore the theories behind these objects much, because the proofs of our results do not require us to use any substantial details beyond what was established in the literature.  However, we will give brief definitions.

\begin{defn}
A vertex algebra over $\bC$ is a complex vector space $V$ equipped with a distinguished vector $\unit \in V$, a distinguished linear transformation $T: V \to V$, and a left multiplication map $Y: V \to (\End V)[[z,z^{-1}]]$, written $Y(a,z) = \sum_{n \in \bZ} a_n z^{-n-1}$ satisfying the following conditions:
\begin{enumerate}
\item $Y(\unit, z) = id_V = id_V z^0$, and $Y(a,z)\unit \in a + zV[[z]]$.
\item For any $a,b \in V$, $Y(a,z)b \in V((z))$, i.e., $a_n b = 0$ for $n$ sufficiently large.
\item $[T, Y(a,z)] = \frac{d}{dz} Y(a,z)$
\item The Jacobi identity: for any $a,b \in V$, 
\[ x^{-1} \delta\left(\frac{y-z}{x}\right) Y(a,y) Y(b,z) - x^{-1} \delta\left(\frac{z-y}{-x}\right) Y(b,z) Y(a,y) =  
z^{-1} \delta\left(\frac{y-x}{z}\right) Y(Y(a,x)b,z) \]
where $\delta(z) = \sum_{n \in \bZ} z^n$ and $\delta(\frac{y-z}{x})$ is expanded as a formal power series in $z$.
\end{enumerate}
An automorphism of a vertex algebra $V$ is a linear transformation $g: V \to V$ that fixes $\unit$, and satisfies $Y(ga,z)gb = gY(a,z)b$ for all $a,b \in V$.
\end{defn}

\begin{defn}
A conformal vertex algebra of central charge $c \in \bC$ is a vertex algebra equipped with a distinguished nonzero vector $\omega$, satisfying the following conditions:
\begin{enumerate}
\item If we write $Y(\omega,z) = \sum_{n \in \bZ} L_n z^{-n-2}$, then the coefficients $L_n \in \End V$ satisfy the Virasoro relations:
\[ [L_m, L_n] = (m-n)L_{m+n} + c \frac{m^3-m}{12} \delta_{m+n,0} id_V \]
\item $L_0$ acts semisimply on $V$, with integer eigenvalues.
\item $L_{-1} = T$.
\end{enumerate}
An automorphism of a conformal vertex algebra is an automorphism of the underlying vertex algebra that fixes $\omega$.  A vertex operator algebra is a conformal vertex algebra for which the eigenvalues of $L_0$ are bounded below and have finite multiplicity.
\end{defn}

Vertex algebras have a more complicated representation theory than ordinary algebras, because modules may be twisted by automorphisms of the vertex algebra.  Here and for the remainder of the paper, we use the notation $e(x)$ to denote the normalized exponential $e^{2\pi i x}$.

\begin{defn}
Let $V$ be a vertex algebra over $\bC$, and let $g$ be an automorphism of order $n$.  Then a $g$-twisted $V$-module is a complex vector space $M$ equipped with an action map $Y^M: V \to (\End M)[[z^{1/n}, z^{-1/n}]]$, written $Y^M(a,z) = \sum_{k \in \frac{1}{n} \bZ} a_k z^{-k-1}$ satisfying the following conditions:
\begin{enumerate}
\item $Y^M(\unit,z) = id_M$
\item For any $a \in V$, $u \in M$, $Y^M(a,z)u \in M((z^{1/n}))$, i.e., $a_k u = 0$ if $k$ is sufficiently large.
\item If $a \in V$ satisfies $ga = e(k/n)a$ for some $k \in \bZ$, then $Y^M(a,z) \in z^{k/n}(\End M)[[z,z^{-1}]]$.
\item The twisted Jacobi identity holds: if $ga = e(k/n)a$, then for any $b \in V$,
\[ \begin{aligned}
x^{-1} \delta\left(\frac{y-z}{x}\right) &Y^M(a,y) Y^M(b,z) - x^{-1} \delta\left(\frac{z-y}{-x}\right) Y^M(b,z) Y^M(a,y) =  \\
&= z^{-1} \left(\frac{y-x}{z} \right)^{-k/n} \delta\left(\frac{y-x}{z}\right) Y^M(Y(a,x)b,z)
\end{aligned} \]
\end{enumerate}
If $V$ is a vertex operator algebra, these conditions define the notion of ``weak $g$-twisted $V$-module''.  An ``admissible $g$-twisted $V$-module'' is a weak $g$-twisted $V$-module that admits a $\frac{1}{n}\bZ_{\geq 0}$-grading that is compatible with the $L_0$-grading on $V$.  An ``ordinary $g$-twisted $V$-module'' is an admissible $g$-twisted $V$-module for which $L_0$ acts semisimply, with finite dimensional eigenspaces, and eigenvalues that are bounded below in each coset of $\bZ$ in $\bC$.  When $g=1$, we replace ``$g$-twisted $V$-module'' with ``$V$-module''.
\end{defn}

\begin{defn}
Here are a few more technical conditions about vertex algebras.  We won't use them in any specific way, but they appear as conditions in the theorems we need.
\begin{enumerate}
\item A vertex algebra $V$ is $C_2$-cofinite if the subspace $C_2(V)$ spanned by $\{ a_{-2} b | a,b \in V\}$ satisfies the property that $V/C_2(V)$ is finite dimensional.
\item A vertex operator algebra $V$ is holomorphic if all admissible $V$-modules are isomorphic to direct sums of $V$.
\item A vertex operator algebra $V$ is of CFT type if $L_0$ has only non-negative eigenvalues, and the kernel of $L_0$ is spanned by $\unit$.
\end{enumerate}
\end{defn}

\begin{defn} (\cite{DL93})
Let $A$ be an abelian group, and let $(F,\Omega)$ be a normalized abelian 3-cocycle on $A$, i.e., a pair of maps $F:A^{\oplus 3} \to \bC^\times$, $\Omega: A^{\oplus 2} \to \bC^\times$ satisfying:
\begin{enumerate}
\item (pentagon) $F(i,j,k) F(i,j+k,\ell) F(j,k,\ell) = F(i+j,k,\ell) F(i,j,k+\ell)$
\item (hexagon) $F(i,j,k)^{-1}\Omega(i,j+k) F(j,k,i)^{-1} = \Omega(i,j) F(j,i,k)^{-1} \Omega(i,k)$
\item (hexagon) $F(i,j,k) \Omega(i+j,k) F(k,i,j) = \Omega(j,k) F(i,k,j) \Omega(i,k)$
\item (normalization) $F(i,j,0) = F(i,0,k) = F(0,j,k) = 1$ and $\Omega(i,0) = \Omega(0,j) = 1$
\end{enumerate}
for all $i,j,k,\ell \in A$.  We define the bilinear form $b_\Omega: A \times A \to \bC/\bZ$ by $e(b_\Omega(a,b)) = \frac{\Omega(a+b,a+b)}{\Omega(a,a)\Omega(b,b)}$.  Then an abelian intertwining algebra of level $N \in \bZ_{\geq 1}$ and central charge $c \in \bQ$ associated to the datum $(A, F, \Omega)$ is a complex vector space $V$ equipped with
\begin{enumerate}
\item a $\frac{1}{N}\bZ \times A$-grading $V = \bigoplus_{n \in \frac{1}{N} \bZ} V_n = \bigoplus_{i \in A} V^i = \bigoplus_{n \in \frac{1}{N} \bZ, i \in A} V_n^i$
\item a left-multiplication $Y: V \to (\End V)[[z^{1/N},z^{-1/N}]]$ written $Y(a,z) = \sum_{n \in \frac{1}{N}\bZ} a_n z^{-n-1}$, and
\item distinguished vectors $\unit \in V_0^0$ and $\omega \in V_2^0$
\end{enumerate}
satisfying the following conditions for any $i,j,k \in A$, $a \in V^i$, $b \in V^j$, $u \in V^k$, and $n \in \frac{1}{N}\bZ$:
\begin{enumerate}
\item $a_n b \in V^{i+j}$.
\item $a_n b = 0$ for $n$ sufficiently large.
\item $Y(\unit, z)a = a$
\item $Y(a,z)\unit \in a + zV[[z]]$.
\item  $Y(a,z)b = \sum_{k \in b_\Omega(i,j) + \bZ} a_k b z^{-k-1}$.
\item The Jacobi identity:
\[ \begin{aligned}
 x^{-1} &\left(\frac{y-z}{x} \right)^{b_\Omega(i,j)} \delta\left(\frac{y-z}{x}\right) Y(a,y) Y(b,z) u \\
 &- B(i,j,k) x^{-1} \left(\frac{z-y}{e^{i\pi}x} \right)^{b_\Omega(i,j)} \delta\left(\frac{z-y}{-x}\right) Y(b,z) Y(a,y) u = \\
& \qquad = F(i,j,k) z^{-1} \delta\left(\frac{y-x}{z}\right) Y(Y(a,x)b,z) \left(\frac{y-x}{z} \right)^{-b_\Omega(i,k)} u 
\end{aligned}\]
where $B(i,j,k) = \frac{\Omega(i,j)F(i,j,k)}{F(j,i,k)}$. 
\item The coefficients of $Y(\omega,z) = \sum_{k \in \bZ} L_k z^{-k-2}$ satisfy the Virasoro relations at central charge $c$.
\item $L_0 a = na$ if $a \in V_n$.
\item $\frac{d}{dz} Y(a,z) = Y(L_{-1}a,z)$.
\end{enumerate}
The type of a quadratic form is the pair $(A,q)$, where $q: A \to \bC/\bZ$ is the unique quadratic form satisfying $e(q(a)) = \Omega(a,a)$ for all $a \in A$.
\end{defn}

Abelian intertwining algebras of type $(A,q)$ form a category $AIA_{(A,q)}$ - see \cite{GM4} section 2.2 for details.

The standard example of an abelian intertwining algebra, described in detail in Chapter 12 of \cite{DL93}, is that of a rational lattice, i.e., a free abelian group $L$ of finite rank equipped with a non-degenerate $\bQ$-valued quadratic form $Q$.  By non-degeneracy of $Q$, there is some $r,s \in \bZ_{\geq 0}$ such that $L \otimes_{\bZ} \bR \cong \bR^{r,s}$.  The corresponding abelian intertwining algebra $V_L$ is constructed as a sum $\bigoplus_{\lambda \in L} \pi^{r,s}_\lambda$ of irreducible modules of the Heisenberg Lie algebra $(L \otimes_{\bZ} \bC[z,z^{-1}]) \oplus \bC K$.  Here, $\pi^{r,s}_\lambda$ is generated by a vector $|\lambda\rangle$ on which $L \otimes_{\bZ}z\bC[z]$ acts by zero, elements $\mu \in L \otimes z^0$ act by the scalar $(\mu,\lambda)$, and $K$ acts by identity.  Multiplication operations between Heisenberg modules are given by intertwining operators, which are unique up to a constant.  Any normalized choice of intertwining operators then yields a suitable 3-cocycle.

\section{Fixed-point free automorphisms yielding no massless states} \label{sect:lifts}

We consider the vertex operator algebra $V_\Lambda$ attached to the Leech lattice $\Lambda$.  If $g$ is a fixed-point free automorphism of $\Lambda$, then all lifts of $g$ to automorphisms of $V_\Lambda$ are conjugate to each other.  I believe this is well-known to experts, but I couldn't find anything close to an explicit statement in the literature before this year.

\begin{prop} (\cite{LS17})
Let $g$ be an automorphism of $\Lambda$ such that $(\Lambda \otimes \bC)^g = \{0\}$.  Then a lift of $g$ to an automorphism of $V_\Lambda$ exists, and for any two lifts $\hat{g}, \tilde{g}$ of $g$ to automorphisms of $V_\Lambda$, there exists some $h \in \Aut V_\Lambda$ such that $h\hat{g} = \tilde{g}h$.
\end{prop}
\begin{proof}
This follows immediately from the results of \cite{LS17} section 4.2, but we offer a brief outline of the reasoning.  By \cite{L85}, any finite order automorphism admits a lift.  By Theorem 2.1 of \cite{DN98}, for any two lifts $\hat{g}, \tilde{g}$ of $g$ to automorphisms of $V_\Lambda$, there exists $a \in \Lambda \otimes \bC$ such that $\hat{g} = e^{a_0}\tilde{g}$, and exponentials satisfy the commutation relation $e^{a_0} \tilde{g} = \tilde{g}e^{g^{-1}(a)_0}$.  Since $g$ is fixed-point free, there is some $b \in \Lambda \otimes \bC$ such that $g(b)-b = a$, e.g., we may take $b = \sum_{i=0}^{|g|-1} \frac{i}{|g|}g^i(a)$.  Thus,
\[ \hat{g} = e^{a_0}\tilde{g} = e^{-b_0}e^{g(b)_0}\tilde{g} = e^{-b_0} \tilde{g} e^{b_0} \]
and the two lifts are conjugate by a lift of the identity on $\Lambda$.
\end{proof}

Naturally, this proof extends to fixed-point free automorphisms of any even lattice that has no roots.

In section 3.3 of \cite{T93}, we are introduced to 3 conditions on automorphisms $a$ of $\Lambda$.  Letting $n$ be the order of the automorphism $a$, we decompose the characteristic polynomial of $a$ as a product $\det(xI_{24} - a) = \prod_{k|n} (x^k-1)^{a_k}$ for uniquely defined integers $a_k$.  Following \cite{F72}, we define the Frame shape of $a$ as the ``generalized permutation'' $\prod_{k|n} k^{a_k} = 1^{a_1}\cdots n^{a_n}$.  Then we define the ``$a$-twisted vacuum energy'' $E_0^a = -\sum_{k|n} \frac{a_k}{k}$, and consider the following three conditions:
\begin{enumerate}
\item (fixed-point free) $\sum_{k|n} a_k = 0$ 
\item (no massless states) $E_0^a > 0$
\item (anomaly-free) $E_0^a \in \frac{1}{n}\bZ$
\end{enumerate}

By a short calculation, one can show that the first condition is equivalent to $\dim (\Lambda \otimes \bC)^a = 0$, and that $E_0^a + 1$ is the lowest $L_0$-eigenvalue of the irreducible $a$-twisted $V_\Lambda$-module $V_\Lambda(a)$.  In order for the orbifold $V_\Lambda/a$ to be isomorphic to $V^\natural$, and in particular to have no weight 1 space, it is necessary that $a$ be fixed-point free (giving no contribution from $V_\Lambda$) and that $E_0^a > 0$ (giving no contribution from the $a$-twisted module $V_\Lambda(a)$).  In fact, it is also necessary that all other nontrivially twisted modules yield no weight 1 contribution, meaning the lift of $a$ on $V_\Lambda(a^j)$ acts fixed-point freely on the weight 1 subspace.  I do not know a way to rewrite this additional condition in terms of Frame shape, but we will find experimentally that it does not eliminate any new classes.

A brief examination of the character table of $Co_0$ yields 165 conjugacy classes of elements, which form 160 algebraic conjugacy classes (where classes are algebraically conjugate if a Galois automorphism of a suitable cyclotomic extension of $\bQ$ permutes the corresponding columns of the character table).  By examining frame shapes in \cite{K85}, we find that there are 95 conjugacy classes of fixed-point free elements, which form 90 algebraic conjugacy classes.  A more detailed examination by Tuite showed that there are 53 conjugacy classes satisfying the first two conditions, which form 51 algebraic conjugacy classes.  Furthermore, 40 conjugacy classes satisfy all three conditions, and they form 38 algebraic conjugacy classes.  We note that Frame shapes only detect algebraic conjugacy.

If we define the modular forms $\eta_a(\tau)$ by $\prod_{k|n} \eta(k\tau)^{a_k}$ and $\theta_a(\tau)$ by the theta function of the fixed-point sublattice $\Lambda^a$, then in section 7 of \cite{CN79}, we find the claim that $\theta_a(\tau)/\eta_a(\tau)$ is the McKay-Thompson series of some element of $\bM$.  In this paper, we are only concerned with the special case that $a$ is fixed-point free, so we simply consider $1/\eta_a(\tau)$.  One may find all of the relevant Frame shapes in Tables 1 and 2 of \cite{T93}, and one may see that they are the reciprocals of the corresponding eta expansions of McKay-Thompson series in the appendix to this paper.
tui
\section{Anomalous orbifolds} \label{sect:anomalous}

We recall from Theorem 10.3 of \cite{DLM97} that if $V$ is a simple, $C_2$-cofinite, holomorphic vertex operator algebra $V$, and $g$ is an automorphism of finite order $n$, then there is a unique $g$-twisted $V$-module, up to isomorphism (which we will call $V(g)$), and its $L(0)$-spectrum lies in $\frac{t}{n^2} + \frac{1}{n}\bZ$ for some uniquely determined $t \in \bZ/n\bZ$.  We say that $g$ is anomalous if $t \neq 0$, and we say that $g$ is anomaly-free if $t = 0$.

The key construction for this paper is the following: By \cite{vEMS} Theorem 5.15, if $V$ is a simple, $C_2$-cofinite, holomorphic vertex operator algebra $V$ of CFT type, and $g$ is an automorphism of finite order $n$, such that the nontrivial irreducible twisted modules $V(g^i)$ have strictly positive $L(0)$-spectrum, then there is an abelian intertwining algebra structure on ${}^gV = \bigoplus_{i=0}^{n-1} V(g^i)$, graded by an abelian group $D$ that lies in an exact sequence $0 \to \bZ/n\bZ \to D \to \bZ/n\bZ \to 0$, with addition law determined by the ``add with carry'' 2-cocycle
\[ c_{2t}(i,k) = \begin{cases} 0 & i_n + k_n < n \\ 2t & i_n + k_n \geq n \end{cases} \]
where the notation $i_n$ denotes the unique representative of $i \in \bZ/n\bZ$ in $\{0,\ldots,n-1\}$.  By \textit{loc. cit.} Proposition 5.13, the quadratic form on $D$ is isomorphic to the discriminant form on the even lattice with Gram matrix $\left( \begin{smallmatrix} -2t_n & n \\ n & 0 \end{smallmatrix} \right)$.

Furthermore, by Theorem 5.16, if $t=0$ (i.e., $g$ is anomaly-free), then the abelian intertwining algebra ${}^gV$ is naturally graded by $\bZ/n\bZ \times \bZ/n\bZ$, such that $V$ is the sum of the degree $(0,i)$ pieces, and that there is a simple $C_2$-cofinite, holomorphic vertex operator algebra $V/g$ of CFT type given by the sum of the degree $(j,0)$ pieces (for $0 \leq j < n$).  The natural $\bZ/n\bZ$-grading from this decomposition endows $V/g$ with a canonical automorphism $g^*$ whose order is equal to $|g|$, such that $(V/g)/g^* \cong V$ and $g^{**} = g$.  We remove the condition on $t$ in this theorem using the following unrolling construction given in \cite{GM4}.

\begin{defn} (\cite{GM4} Definition 2.3.3)
Given a diagram $(A, q) \overset{\pi}{\twoheadleftarrow} (A', q') \overset{i}{\hookrightarrow} (A'',q'')$ of quadratic spaces, we define the unrolling functor $AIA_{(A,q)} \to AIA_{(A'',q'')}$ to be the composite of the following two functors:
\begin{itemize}
\item Pullback along $\pi$: For any $a' \in A'$, the degree $a'$ part of the pullback abelian intertwining algebra is equal to the degree $\pi(a')$ part of the source abelian intertwining algebra.  Multiplication is defined in the obvious way.
\item Extension by zero: The degree $a''$ part of the target abelian intertwining algebra is zero if $a'' \notin i(A')$, and equal to the degree $a'$ part of the source abelian intertwining algebra if $a'' = i(a')$.  Multiplication is defined in the obvious way.
\end{itemize}
\end{defn}

\begin{prop} (\cite{GM4} Proposition 2.3.4)
Let $V$ be a simple, $C_2$-cofinite, holomorphic vertex operator algebra $V$ of CFT type, and let $g$ be an automorphism of finite order $n$, such that the nontrivial twisted modules $V(g^i)$ have strictly positive conformal weight.  Let $N = \frac{n^2}{(n,t)}$, where $t$ is given in the introduction of the section.  Then we may unroll ${}^g V$ along the diagram
\[ (D, q) \twoheadleftarrow \langle (1, \frac{t_n}{(n,t)}), (0, \frac{n}{(n,t)}) \rangle \hookrightarrow \bZ/N\bZ \times \bZ/N\bZ \]
where the first arrow is given by $a(1, \frac{t_n}{(n,t)}) + b(0, \frac{n}{(n,t)}) \mapsto (a_n, b_n + 2t_n\lfloor a/n \rfloor)$, and the quadratic form on $\bZ/N\bZ \times \bZ/N\bZ$ is given by $(a,b) \mapsto \frac{ab}{N}$.  Unrolling then
yields a $\bZ/N\bZ \times \bZ/N\bZ$-graded abelian intertwining algebra structure on ${}^g_N V = \bigoplus_{i=0}^{N-1} V(g^i)$, where the first copy of $\bZ/N\bZ$ parametrizes twisting, and the second is determined by the eigenvalues of the unique homogeneous automorphism $\tilde{g}$ of order $N$ that acts on $V$ by $g$, and on $V(g)$ by $e(L_0) = e^{2\pi i L_0}$.  Furthermore, if the central charge of $V$ is a multiple of 24, then the characters of these graded pieces of ${}^g_N V$ form a vector-valued modular function $F^g = \{F^g_{i,j}(\tau)\}_{i,j \in \bZ/N\bZ}$ of type $\rho_{I\!I_{1,1}(N)}$.
\end{prop}

We then define the generalized orbifold construction:

\begin{defn} \label{defn:cyclic-orbifold}
Let $V$ be a simple, $C_2$-cofinite, holomorphic vertex operator algebra $V$ of CFT type, and let $g$ be an automorphism of finite order $n$, such that the nontrivial twisted modules $V(g^j)$ have strictly positive conformal weight.  We define $V/g = ({}^g_NV)^{\tilde{g}} = \bigoplus_{j=0}^{N-1} {}^g_N V_{(j,0)}$, and we define $g^*$ to be the automorphism that takes any vector $v \in V(g^j)$ to $e(j/N)v = e^{2\pi i j/N}v$.
\end{defn}

We now show that anomalous orbifolds do not produce any new vertex operator algebras.  While we do not get new objects, the construction is still useful, because we obtain a correspondence of automorphisms.
%Since $g$ has order $n$, $g^{N/n} = g^{n/(n,t)}$ has order $n^2/N = (n,t)$.

\begin{prop} \label{prop:anomalous-orbifold}
The smallest positive power of $g$ that is anomaly-free is $g^{n/(n,t)} = g^{N/n}$, and $V/g$ is a simple, $C_2$-cofinite, holomorphic vertex operator algebra of CFT type that is isomorphic to $V/g^{N/n}$.  Furthermore, $g^*$ restricts to an automorphism of order $n$ on $V/g$ that further restricts to $g$ on $V^{g^{N/n}} = (V/g)^{(g^*)^{N/n}}$.  Finally, $(V/g)/g^* \cong V$ and $g^{**} = g$.
\end{prop}
\begin{proof}
The claim about the anomaly-free power follows straightforwardly from considering the quadratic form (and hence conformal weights) on ${}^g_NV$.  From the structure of our unrolling, we see that the summand ${}^g_N V_{(j,0)}$ making up $V/g$ is nonzero if and only if $j\frac{t_n}{(n,t)}\in \frac{n}{(n,t)} \bZ$, i.e., $j \in \frac{n}{(n,t)}\bZ$.  Thus, we have $n$ summands, giving the order of $g^*$, and the claim about the restriction is clear.  To prove the isomorphism $V/g \cong V/g^{N/n}$, we apply the uniqueness of simple current extensions from Proposition 5.3 of \cite{DM02a}, so it suffices to show that the irreducible $V^g$-modules making up $V/g$ are isomorphic to the irreducible $V^g$-modules making up $V/g^{N/n}$.  

In the course of unrolling, we replicate the twisted modules, but with a shifted $\tilde{g}$-action.  The shifting is determined uniquely by the kernel of the pullback map, which is generated by $(n,\frac{-nt_n}{(n,t)}) \in \bZ/N\bZ \times \bZ/N\bZ$.  In particular, the module labeled $V(g^n)$ is isomorphic to $V$ as a $V$-module, but where the action of $g$ is shifted by the constant $e(-\frac{t}{n})$.  We therefore have $V^g$-module isomorphisms identifying $V(g^{jn+k})^{\tilde{g}=\zeta}$ with $V(g^k)^{\tilde{g} = e(-jt/n)\zeta}$ for all constants $\zeta$, so the sum defining $V/g$ is identified with $\bigoplus_{j=0}^{n/(n,t)-1} \bigoplus_{k=0}^{(n,t) - 1} V(g^{nk/(n,t)})^{\tilde{g} = e(-jt/n)}$, and this is equal to the sum defining $V/g^{n/(n,t)}$.
\end{proof}

\begin{cor} \label{cor:orbifold-isomorphism-for-abelian-intertwining-algebras}
Let $V$ be a simple, $C_2$-cofinite, holomorphic vertex operator algebra $V$ of CFT type, and let $g$ be an automorphism of finite order.  Then the automorphism of $\bZ/N\bZ \times \bZ/N\bZ$ given by $(a,b) \mapsto (b,a)$ preserves the quadratic form $q: (a,b) \mapsto \frac{ab}{N}$, and this automorphism on grading groups induces an isomorphism ${}^g_NV \cong {}^{g^*}_N (V/g)$ in $AIA_{(\bZ/N\bZ \times \bZ/N\bZ, q)}$.
\end{cor}
\begin{proof}
It suffices to show that ${}^g_NV$, under the switched grading, satisfies the defining properties of ${}^{g^*}_N(V/g)$.  It is immediate that we have an abelian intertwining algebra of the correct type, and that $\bigoplus_{j=0}^{N-1} (^g_NV)_{(0,j)}$ is isomorphic to $V/g$.  It remains to show that for each $1 \leq i < N$, the sum $\bigoplus_{j=0}^{N-1} (^g_NV)_{(i,j)}$ is an irreducible $(g^*)^i$-twisted $V/g$-module.  However, this follows immediately from restricting the Jacobi identity for abelian intertwining algebras to the appropriate graded pieces.
\end{proof}

\section{Non-Fricke orbifolds of Moonshine} \label{sect:non-fricke}

We examine the McKay-Thompson series of non-Fricke elements of the Monster, and their product expansions, building on the analysis in \cite{GM2}.

Let us review the description of McKay-Thompson series that was initially conjectured in \cite{CN79}, and proved in \cite{B92} to hold for the moonshine module $V^\natural$ constructed in \cite{FLM88}.  For each element $g \in \bM$, the graded character $T_g(\tau) = \Tr(gq^{L_0-1}|V^\natural)$ is a Hauptmodul for some genus zero subgroup $\Gamma_g$ of $SL_2(\bR)$, and we refer to it using the notation $n|h+e_1,e_2,\ldots$, where $n$ is the order of $g$, $h$ is the smallest divisor of $n$ and $24$ such that $\Gamma_g$ contains $\Gamma_0(nh)$, and $e_1,e_2,\ldots$ are exact divisors of $n/h$.  If $n/h$ appears on this list of exact divisors, or if $n=h$, then we say that $g$ is Fricke (i.e., the Fricke involution $\tau \mapsto \frac{-1}{nh\tau}$ acts by $\pm 1$ on $T_g(\tau)$), and if not, then we say that $g$ is non-Fricke.

By \cite{GM4} Proposition 2.3.4 (which is basically \cite{vEMS} Theorem 5.14 with an unrolling adjustment), the abelian intertwining algebra ${}^g_NV^\natural$ has character given by a vector-valued modular function $F^g = \{F^g_{i,j}(\tau)\}_{i,j \in \bZ/N\bZ}$ for the Weil representation $\rho_{I\!I_{1,1}(N)}$ attached to a hyperbolic lattice.  This function has appeared earlier, in \cite{GM2}, where it was lifted to a Borcherds product that described an infinite dimensional Borcherds-Kac-Moody Lie algebra $L_g$ (written as $W_g$ in that paper).  The root multiplicities of $L_g$ are coefficients of $F^g$, and we will use the fact that the coefficients are non-negative integers to determine $V^\natural/g$ for all non-Fricke classes in $\bM$.  We note that by Theorem 1 of \cite{PPV17}, $V^\natural/g \cong V^\natural$ for all anomaly-free Fricke elements $g$ (hence also all anomalous Fricke elements, by Proposition \ref{prop:anomalous-orbifold}), so this completes the determination of cyclic orbifolds.  As a consequence, Tuite's explanation of the genus zero property of moonshine functions is no longer conditional on uniqueness conjectures \cite{T93}.  

\begin{defn} \label{defn:non-negative}
A eta product of the form $\prod_{i=1}^k \eta(a_i \tau)^{b_i}$ is non-negative if the resulting product expansion $q^{\sum_i a_i b_i/24}\prod_{j \geq 1} (1-q^j)^{c_j}$ has $c_j \geq 0$ for all $j \geq 1$.
\end{defn}

\begin{prop} \label{prop:unique-non-negative-eta-product}
Let $g$ be a non-Fricke element of $\bM$.  Then there is a unique non-negative eta product expansion of the McKay-Thompson series $T_g(\tau)$.
\end{prop}
\begin{proof}
By \cite{GM2} Theorem 3.24 and Corollary 3.25, we have the following product expansion:
\[ T_g(\sigma) - T_g(-1/\tau) = p^{-1} \prod_{i>0, j \in \frac{1}{N}\bZ} (1-p^i q^j)^{c^g_{i,j}(ij)} \]
where $c^g_{i,j}(ij)$ is the coefficient of $q^{ij-1}$ in the $q$-expansion of $F^g_{i,j}$.  Because $T_g$ is non-Fricke, $T_g(-1/\tau)$ is regular at infinity, so this product expansion has no negative powers of $q$.  By taking the limit $q \to 0$, we find that
\[ T_g(\sigma) - \lim_{\tau \to 0} T_g(\tau) = p^{-1} \prod_{i>0} (1-p^i)^{c^g_{i,0}(0)} \]
By Theorem 2.5.4 of \cite{GM4} (essentially combining the previously mentioned results of \cite{vEMS} with manipulations in \cite{GM2}), each $c^g_{i,0}(0)$ is the dimension of a subspace of $V^\natural(g^i)$, so it is a non-negative integer.  We thus obtain a non-negative eta product for $T_g$.  Uniqueness follows from Lemma 3.6 in \cite{GM2}, which asserts that $F^g$, hence the set of its constant terms, is completely determined by the McKay-Thompson series of the powers of $g$.
\end{proof}

\begin{rem}
Table 3 in \cite{CN79} gives eta product expansions of many McKay-Thompson series, but some elements are given multiple eta product expansions, with no particular reason to choose one over the others.  Proposition \ref{prop:unique-non-negative-eta-product} gives us a way to make a distinguished choice, namely the non-negative expansions.
\end{rem}

\begin{rem}
It would be nice to have a characterization of non-Fricke completely replicable functions as eta products.  Many of the properties of the non-negative eta products, such as the balance between exponents on $\eta(a\tau)$ and $\eta(\frac{N}{a}\tau)$, seem to follow naturally from the Hauptmodul property applied to the expansion at zero.
\end{rem}

\begin{thm} \label{thm:main}
If $g$ is non-Fricke, then $V^\natural/g$ is isomorphic to the Leech lattice vertex operator algebra $V_\Lambda$, and $g^*$ induces a fixed-point free automorphism of the Leech lattice $\Lambda$ satisfying the ``no massless states'' condition.  If $\sigma$ is a fixed-point free automorphism of the Leech lattice $\Lambda$ satisfying the ``no massless states'' condition, then there is a lift $\tilde{\sigma}$ of $\sigma$ to an automorphism of $V_\Lambda$, unique up to conjugation, such that $V_\Lambda/\tilde{\sigma} \cong V^\natural$, and $\sigma^*$ is non-Fricke.
\end{thm}
\begin{proof}
We first note that by Theorem 2.2.9 of \cite{GM4}, the hypotheses of Definition \ref{defn:cyclic-orbifold} apply to any automorphism of $V^\natural$, so we may apply the cyclic orbifold theory established in \cite{vEMS}.  By Proposition \ref{prop:anomalous-orbifold}, $V^\natural/g$ is a simple, $C_2$-cofinite, holomorphic vertex operator algebra $V^\natural/g$ with central charge 24 and of CFT type, and it is isomorphic to the anomaly-free orbifold $V^\natural/g^{N/n}$.  The weight 1 subspace of $V^\natural/g^{N/n}$ has dimension equal to the sum of constant terms $\sum_{i=0}^{(n,t)-1} c^{g^{N/n}}_{i,0}(0)$.  If $T_{g^{N/n}}(\tau)$ is given by the non-negative eta product $\prod_{i=1}^k \eta(a_i \tau)^{b_i}$, then the constant terms of $\{ F_{j,0} \}_{j \in \bZ/(n,t)\bZ}$ are given by $\sum_{a_i|j} b_i$.  Using these facts, we have enumerated the non-negative eta products of the anomaly-free non-Fricke classes $g$ in the Appendix, and found that the constant terms always sum to 24.  By the uniqueness result of \cite{DM02b}, any simple $C_2$-cofinite holomorphic vertex operator algebra of central charge 24 with 24-dimensional weight one subspace is isomorphic to $V_\Lambda$.  The claim about $g^*$ then follows from the fact that $V_\Lambda/g^*$ has dimension zero weight 1 space, eliminating the possibility of fixed points and massless states.

To show the claim about $\sigma$, it suffices to enumerate the relevant classes, and see that all have the form $g^*$ for some non-Fricke $g$.  This is done for the 38 non-anomalous classes in the Appendix, and for the 13 anomalous classes in Table 2 of \cite{T93}.
\end{proof}

\begin{cor}
The cyclic orbifold correspondence between $V^\natural$ and $V_\Lambda$ induces a natural bijection between algebraic conjugacy classes of non-Fricke elements of $\bM$ and algebraic conjugacy classes of fixed-point free automorphisms of $\Lambda$ satisfying the ``no massless states'' condition.  This bijection respects the anomaly-free property, and induces isomorphisms between quotients of centralizers by the centralizing elements.
\end{cor}
\begin{proof}
From Theorem \ref{thm:main}, we obtain a correspondence between non-Fricke $g \in \bM$ and fixed-point free $g^* \in \Aut V_\Lambda$ satisfying the ``no massless states'' condition.  Inside the automorphism group of $(V^\natural)^g$, which is identified with $(V_\Lambda)^{g^*}$, we have the group of all automorphisms that admit lifts to homogeneous automorphisms of the abelian intertwining algebra ${}^gV_\Lambda$.  This group is identified with the quotients $C_{\bM}(g)/\langle g \rangle$ and $C_{\Aut V_\Lambda}(g^*)/\langle g^* \rangle$, respectively, so the groups are isomorphic.
\end{proof}

\section{Monstrous Lie algebras} \label{sect:monstrous}

In sections 2 and 3 of \cite{GM4}, we constructed an infinite dimensional Lie algebra $\fm_g$ for each $g \in \bM$ by applying the composite of an ``add a torus'' functor $AT^i_L$ (where $i: (L^\vee/L, e(Q)) \to (D,q)$ is an orbifold-admissible quadratic isomorphism) and a bosonic string quantization functor $OCQ$ to the abelian intertwining algebra ${}^g V^\natural$.  In Proposition 3.4.3 of \textit{loc. cit.}, we showed that if $g$ is Fricke, then $\fm_g$ is a Borcherds-Kac-Moody Lie algebra (defined in \cite{B88}), by appealing to a sufficient set of conditions adapted from \cite{B95}.  However, for the case that $g$ is non-Fricke, we were unable to verify that one of the conditions holds.  The Fricke case was sufficient to prove the Generalized Moonshine conjecture, so we left it at that.  However, with the results of the previous sections, we can prove that the missing condition also holds in the non-Fricke case.

We begin by defining Borcherds-Kac-Moody Lie algebras.

\begin{defn}
If $I$ is a countable index set, a matrix $A = (a_{i,j})_{i,j \in I}$ of real numbers is called a generalized Cartan matrix if it satisfies the following conditions:
\begin{enumerate}
\item $A$ is symmetrizable, i.e., there is a diagonal matrix $Q$ whose diagonal entries $Q_{i,i} = q_i$ are positive real numbers, such that $QA$ is symmetric.
\item $a_{i,j} < 0$ if $i \neq j$.
\item If $a_{i,i} > 0$, then $a_{i,i} = 2$ and for all $j \in I$, $a_{i,j} \in \bZ$.
\end{enumerate}
Given a generalized Cartan matrix $(a_{i,j})_{i,j \in I}$, its universal Borcherds-Kac-Moody algebra $\fg(A)$ is the Lie algebra with generators $\{ h_i, e_i, f_i\}_{i,j \in I}$, and the following relations:
\begin{enumerate}
\item $\mathfrak{sl}_2$ relations: $[h_i, e_k] = a_{i,k} e_k$, $[h_i, f_k] = -a_{i,k} f_k$, $[e_i,f_j] = \delta_{i,j} h_i$.
\item Serre relations: If $a_{i,i} > 0$, then $\ad(e_i)^{1-2a_{i,j}}(e_j) = \ad(f_i)^{1-2a_{i,j}}(f_j) = 0$.
\item Orthogonality: If $a_{i,j} = 0$, then $[e_i, e_j] = [f_i,f_j] = 0$.
\end{enumerate}
A Borcherds-Kac-Moody algebra is a Lie algebra of the form $(\fg(A)/C).D$, where $C$ is a central ideal, and $D$ is a commutative Lie algebra of outer derivations.
\end{defn}

\begin{thm} \label{thm:mg-is-bkm}
If $g$ is non-Fricke, then the Monstrous Lie algebra $\fm_g$ is a Borcherds-Kac-Moody Lie algebra.
\end{thm}
\begin{proof}
From the proof of \cite{GM4} Proposition 3.4.3, the only condition that is not immediately verified is given as follows: ``Any two roots of non-positive norm that are both positive or both negative have inner product at most zero, and if the inner product is zero, then the root spaces commute.''

The Lie algebra $\fm_g$ is $I\!I_{1,1}(-1/N)$-graded, where we identify the root lattice $I\!I_{1,1}(-1/N)$ with $\bZ \times \frac{1}{N}\bZ$ with quadratic form $(a,b) \mapsto -ab$, and the space of degree $(a,b) \neq (0,0)$ is called the $(a,b)$-root space.  Positive roots $(a,b)$ of non-positive norm necessarily lie in the region where both $a$ and $b$ are non-negative, so their inner product is at most zero.  The same conclusion then holds for negative roots.  The inner product is zero if and only if both roots have norm zero, i.e., they both lie on the $x$-axis or on the $y$-axis.  Thus, it suffices to show that the Lie bracket on norm zero roots of $\fm_g$ is commutative.

We claim that the Lie bracket on norm zero roots of $\fm_g$ is induced by the Lie bracket on the weight 1 subspace of $V^\natural/g$, which can be identified with the commutative Lie algebra $\Lambda \otimes \bC$ by Theorem \ref{thm:main}.  We show this through explicit analysis of the ``old canonical quantization'' functor, which sends conformal vertex algebras of central charge 26 to Lie algebras, and is defined so that the left multiplication map $a \mapsto [v,a]$ in the Lie algebra is given by the $v_0$ operation in the weight 1 subspace of the conformal vertex algebra.  The Goddard-Thorn no-ghost theorem \cite{GT72} (see the appendix to \cite{J98} for a detailed proof) identifies the subspace of $\fm_g$ in degree $(i,0)$ with the weight 1 subspace of $V^\natural(g^i)^{\tilde{g}} \otimes \bC |(i,0)\rangle$, or equivalently $V_\Lambda^{g^* = e(i/N)} \otimes \bC |(i,0)\rangle$, where $|(i,0)\rangle \in \pi^{1,1}_{(i,0)}$ is the weight zero generating element of the Heisenberg module.  Thus, the Lie bracket is completely determined by the following calculation:
\[ \begin{aligned}
{[}u^1 \otimes |(i,0) \rangle, u^2 \otimes |(j,0)\rangle] &= (u^1 \otimes |(i,0) \rangle)_0 u^2 \otimes |(j,0)\rangle \\
&= \sum_{n \in \bZ} u^1_n u^2 \otimes | (i,0) \rangle_{-n-1} |(j,0) \rangle \\
&= u^1_0 u^2  \otimes |(i,0)\rangle_{-1} |(j,0) \rangle \\
&= u^1_0 u^2  \otimes |(i+j,0) \rangle
\end{aligned} \]
Since the weight one subspace of $V_\Lambda$ is a commutative Lie algebra under the bracket $[u^1,u^2] = u^1_0u^2$, we have $u^1_0 u^2 = 0$, so the norm zero roots commute.
\end{proof}

\begin{cor} \label{cor:lg-isom-mg}
For each non-Fricke $g \in \bM$, the Lie algebra $\fm_g$ is isomorphic to the Lie algebra $L_g$ (named $W_g$ in \cite{GM2} Proposition 4.4).
\end{cor}
\begin{proof}
By Theorem \ref{thm:mg-is-bkm}, both Lie algebras are Borcherds-Kac-Moody, so by \cite{GM4} Lemma 3.4.4, it suffices to show that the root multiplicities match.  For both Lie algebras, the root multiplicity of $(a,b)$ is given by $c^g_{a,b}(ab)$.
\end{proof}

We briefly recall the description of equivariant Hecke operators in \cite{GM1} - see also \cite{T08} and \cite{G07}.  Let $G$ be a finite group, and consider the moduli stack $\cM^G_{Ell}$ of complex-analytic elliptic curves equipped with $G$-torsors.  Then we may define the equivariant Hecke operator $T_n$ as an endomorphism on the space of functions on $\cM^G_{Ell}$ by the formula
\[ T_n f(P \to E) = \frac{1}{n} \sum_{\phi: E' \to E} f(\phi^*P \to E') \]
where the sum runs over all degree $n$ isogenies $\phi: E' \to E$.  The moduli stack is a complex analytic quotient of the form $\Hom(\bZ \times \bZ, G) \overset{SL_2(\bZ)}{\times} \fH$, so a choice of oriented basis of $H_1(E)$ yields a conjugacy class of commuting pairs of elements of $G$, and $f$ can be written $f(g,h;\tau)$, where $g,h\in G$ form a commuting pair.  Then the Hecke operator is written
\[ T_n f(g,h;\tau) = \frac{1}{n}\sum_{ad=n, 0 \leq b < d} f(g^d, g^{-b}h^a; \frac{a\tau+b}{d}) \]
As we explain in \cite{GM1}, this formula also works for the larger moduli space $\Hom(\bZ \times \bZ, G) \overset{\pm \bZ}{\times} \fH$ of elliptic curves equipped with $G$-torsors and multiplicative uniformizations.  This allows us to describe Hecke operators on abstract $q$-expansions without the assumption of modular invariance.

\begin{rem}
The formula for the Hecke operator depends on a sign convention that is not uniform in the literature.  In most of the generalized moonshine literature, we find the implicit use of a uniformization of elliptic curves that identifies the oriented basis of $H_1$ with the pair $(-1,\tau)$ instead of the usual pair $(1,\tau)$.  On a representation theory level, this amounts to the ambiguity between $g$-twisted modules and $g^{-1}$-twisted modules, and the literature is not uniform about which is which.  On the level of functions, this changes the $SL_2(\bZ)$ transformation rule from $f(g^ah^c,g^bh^d;\tau) = f(g,h;\frac{a\tau+b}{c\tau+d})$ to  $f(g^ah^{-c},g^{-b}h^d;\tau) = f(g,h;\frac{a\tau+b}{c\tau+d})$.  The former rule was given in \cite{N87} as part of the Generalized Moonshine conjecture, while the latter was given in \cite{V86} to describe orbifold genus one functions in conformal field theory.
\end{rem}

We say that a function $f$ on $\Hom(\bZ \times \bZ, G) \overset{\pm \bZ}{\times} \fH$ (resp. $\Hom(\bZ \times \bZ, G) \overset{SL_2(\bZ)}{\times} \fH$) is Hecke-monic on $\Hom(\bZ \times \bZ, G) \overset{\pm \bZ}{\times} \fH$ (resp. $\Hom(\bZ \times \bZ, G) \overset{SL_2(\bZ)}{\times} \fH$) if there is a monic polynomial $\Phi_n \in \bQ[x]$ of degree $n$ such that $nT_nf(g,h;\tau) = \Phi_n(f(g,h;\tau))$ for all commuting pairs $(g,h)$ in $G$.

\begin{cor} \label{cor:hecke-monic}
All Generalized Moonshine functions $Z(g,h;\tau) = \Tr(\tilde{h} q^{L_0-1}|V^\natural(g))$ are Hecke-monic on $\Hom(\bZ \times \bZ, \bM) \overset{\pm \bZ}{\times} \fH$.
\end{cor}
\begin{proof}
The Hecke-monic property follows from Proposition 4.8 of \cite{GM2}, once we establish a suitably compatible action of $\widetilde{C_{\bM}(g)}$ on $L_g$.  Such an action is given by \textit{transport de structure} along the isomorphism given in Corollary \ref{cor:lg-isom-mg}.

As we noted in Remark 5.3 of \cite{Fricke}, the proof of Proposition 4.8 in \cite{GM2} has an erroneous line in a string of equalities, but it can be fixed using the reasoning given in the proof of Theorem 5.2 in \cite{Fricke}.
\end{proof}

\section{Refined Generalized Moonshine} \label{sect:generalized}

Let us recall the statement of the Generalized Moonshine conjecture for twisted $V^\natural$ modules, in the form that it was solved in \cite{GM4}:

\begin{thm} \label{thm:gen-moon} (\cite{GM4}, Theorem 4.2.2)
Any rule that assigns:
\begin{itemize}
\item to each $g \in \bM$ the irreducible $g$-twisted $V^\natural$-module $V^\natural(g)$, with its canonical projective $C_{\bM}(g)$-action, and
\item to each commuting pair $(g,h)$ the function $Z(g,h;\tau) = \Tr(\tilde{h} q^{L_0-1}|V^\natural(g))$ for some lift $\tilde{h}$ of $h$ to a finite order linear transformation on $V^\natural(g)$,
\end{itemize}
satisfies the following conditions:
\begin{enumerate}
\item The formal series defining $Z(g,h;\tau)$ is the $q$-expansion of a holomorphic function on the upper half-plane $\fH$.
\item The function $(g,h) \mapsto Z(g,h;\tau)$ is invariant under simultaneous conjugation on the pair $(g,h)$ in $\bM$, up to rescaling.
\item $Z(g,h;\tau)$ is either a constant or a Hauptmodul.
\item For any $\left(\begin{smallmatrix} a & b \\ c & d \end{smallmatrix} \right) \in SL_2(\bZ)$ and any commuting pair $(g,h)$ in $\bM$, $Z(g,h,\frac{a\tau+b}{c\tau+d})$ is proportional to $Z(g^a h^c, g^b h^d,\tau)$.
\item $Z(g,h;\tau)$ is proportional to $J(\tau)$ if and only if $g=h=1$.
\end{enumerate}
\end{thm}

This is more or less the weakest form that can be reasonably called a solution to Norton's conjecture, because there is no control over the ambiguous constants in the simultaneous conjugation rule or the $SL_2(\bZ)$-rule.  Norton suggested in his original statement of the conjecture \cite{N87} that that these ambiguous constants can be refined to roots of unity, and further proposed in \cite{N01} that they can be refined to 24th roots of unity.  About 10 years later, a specific description of these constants in terms of a conjectural ``moonshine element'' $\alpha^\natural \in H^3(\bM, U(1))$ was proposed in \cite{GPV13}.  The group $H^3(\bM, U(1))$ is still a relatively unknown quantity, although there have been suggestions in the folklore for many years that it has order 24 or 48.  We will not prove any of these claims about ambiguous constants in their entirety, but instead only show that Norton's original ``root of unity'' proposal holds for all nonconstant functions.  We follow an argument that was roughly sketched in the final remark of \cite{GM2}.

\begin{thm} \label{thm:refined} (Refined Generalized Moonshine)
The ambiguous scalars in conditions 2 and 4 of Theorem \ref{thm:gen-moon} can be refined to a root-of-unity ambiguity whenever $Z(g,h,\tau)$ is not constant.
\end{thm}
\begin{proof}
We recall that $Z(g,h;\tau)$ was defined as the graded character of a finite order lift $\tilde{h}$ of the centralizing element $h$ to a linear transformation on $V^\natural(g)$.  We showed in \cite{GM4} Theorem 4.2.2 that $Z(g,h;\tau)$ is either a Hauptmodul or a constant function.  To prove the theorem, it suffices to show that if $Z(g,h;\tau)$ is nonconstant, then the leading coefficient of its $q$-expansion at the unique singular cusp is a root of unity.  When $g$ is Fricke, this follows from the Hecke-monic property, by \cite{GM1} Lemma 2.1, and the order of the root of unity is bounded by the least common multiple of $|g|$, $|\tilde{h}|$, and 2.  We note that there is some ambiguity in the notion of leading coefficient.  However, the leading coefficients of different expansions of a modular form at a given cusp only differ by roots of unity.

Let $Z(g,h;\tau)$ be non-constant, with $g$ non-Fricke.  Then by \cite{GM4} Theorem 4.2.2, it is a Hauptmodul invariant under $\Gamma(M)$ for some $M \in \bZ_{>0}$.  Write $\hat{g}$ and $\hat{h}$ for the coordinate generators of $\bZ/M\bZ \times \bZ/M\bZ$, and define a function $Z^s$ on $\Hom(\bZ \times \bZ, \bZ/M\bZ \times \bZ/M\bZ) \overset{SL_2(\bZ)}{\times} \fH$ by
\[ Z^s(\hat{g}^a\hat{h}^c,\hat{g}^b\hat{h}^d;\tau) = \begin{cases} Z(g,h;\frac{a\tau+b}{c\tau+d}) & ad-bc=1 \\ 0 & \text{otherwise} \end{cases} \]
By \cite{CM16} Corollary 6.3, $Z^s(\hat{g}^a\hat{h}^c,\hat{g}^b\hat{h}^d;\tau)$ is proportional to $Z(g^ah^c,g^bh^d;\tau)$, for all $\left( \begin{smallmatrix} a & b \\ c & d \end{smallmatrix} \right) \in SL_2(\bZ)$.

By the $SL_2(\bZ)$-invariance of the definition of equivariant Hecke operators on functions on $\Hom(\bZ \times \bZ, \bZ/M\bZ \times \bZ/M\bZ) \overset{SL_2(\bZ)}{\times} \fH$ given in \cite{GM1} section 1, we find that that if $p$ is a prime congruent to 1 modulo $M$, then $T_pZ^s(\hat{g}^a\hat{h}^c,\hat{g}^b\hat{h}^d;\tau) = T_pZ(g,h;\frac{a\tau+b}{c\tau+d})$ for all $\left( \begin{smallmatrix} a & b \\ c & d \end{smallmatrix} \right) \in SL_2(\bZ)$.  Thus, from the Hecke-monic property of $Z(g,h;\tau)$ given in Corollary \ref{cor:hecke-monic}, we see that there is a monic polynomial $\Phi_p(x) \in \bQ[x]$ of degree $p$ such that $pT_pZ^s(\hat{g}^a\hat{h}^c,\hat{g}^b\hat{h}^d;\tau) = \Phi_p(Z^s(\hat{g}^a\hat{h}^c,\hat{g}^b\hat{h}^d;\tau))$ for all $\left( \begin{smallmatrix} a & b \\ c & d \end{smallmatrix} \right) \in SL_2(\bZ)$.

Since $Z(g, h;\tau)$ is a Hauptmodul and $g$ is non-Fricke, there is some $\left( \begin{smallmatrix} a & b \\ c & d \end{smallmatrix} \right) \in SL_2(\bZ)$ with $c\neq 0$ such that $Z^s(\hat{g}^a\hat{h}^c,\hat{g}^b\hat{h}^d;\tau)$ has a pole at $i\infty$.  We claim that the leading coefficient of the $q$-expansion at that pole is a $\frac{2M}{(M,2)}$th root of unity.  While the statement of \cite{GM1} Lemma 2.1 assumes the input function is weakly Hecke-monic, the argument only uses the $p$th Hecke operators for primes $p$ congruent to 1 modulo $M$, so the result still holds for $Z^s(\hat{g}^a\hat{h}^c,\hat{g}^b\hat{h}^d;\tau)$.  We conclude that the pole of $Z(g,h;\tau)$ at $a/c$ has leading coefficient given by a root of unity, and this ends the proof.
\end{proof}

Recently a preprint \cite{JF17} has appeared with the claim that $\alpha^\natural$ has order 24, and a referee has requested that I comment on its relation to Theorem \ref{thm:refined}.  At the time of this writing, it does not seem that the argument in that preprint is complete, because it rests on still-open conjectures concerning the dictionary relating the orbifold theories of vertex operator algebras and conformal nets.  That said, Johnson-Freyd's preprint is the first reasonably solid argument for a definite value to appear in public, and the open questions on which it depends are expected to have positive answers by essentially all experts.  The arguments in the preprint would yield a substantial strengthening of Theorem \ref{thm:refined} under the following assumptions:

\begin{enumerate}
\item The fixed-point vertex operator subalgebra $(V^\natural)^{\bM}$ is regular.  This would more or less imply the existence of a well-defined 3-cocycle $\alpha$, such that the category of $(V^\natural)^{\bM}$-modules is a modular tensor category equivalent to the $\alpha$-twisted double of the category of $\bM$-modules, and the category of twisted $V^\natural$-modules is an $\bM$-graded fusion category.  I am saying ``more or less'' because I have not seen a proof of the precise implication.
\item The category of $(V^\natural)^{\bM}$-modules is equivalent as a modular tensor category to the category of sectors for the monster fixed-point subnet of the moonshine conformal net.  This subnet is already known to be equivalent to the $\beta$-twisted double for some $\beta$ which is the subject of \cite{JF17}.  The equivalence of categories would imply $\alpha$ and $\beta$ have the same order.  Alternatively, one could try to show that the category of twisted $V^\natural$-modules is equivalent as an $\bM$-graded fusion category to the category of twisted sectors of the moonshine conformal net, and this would canonically identify $\alpha$ and $\beta$.
\item Under the assumption of regularity of the fixed-point vertex operator subalgebra, twisted twining characters of a holomorphic vertex operator algebra satisfy the relations given in \cite{GPV13}.  This does not sound too difficult, but I have only seen a proof of this claim in the case of cyclic groups \cite{vEMS}.
\end{enumerate}

Under the first two assumptions, the calculations in \cite{JF17} would give a description of the vertex-operator-algebraic anomaly, and under the third assumption, the description of ambiguous constants given in \cite{GPV13} together with the claim that they could be taken as 24th roots of unity would then be verified.

\section{Conway Lie algebras of rank 2} \label{sect:conway}

We introduce rank 2 analogues of the Monstrous Lie algebras attached to fixed-point free automorphisms of $V_\Lambda$, and show that they are Borcherds-Kac-Moody.  In fact they are isomorphic to the corresponding Monstrous Lie algebras, once we switch coordinates in the Cartan subalgebra.

\begin{defn}
Let $g$ be a non-Fricke element in $\bM$, and let $g^*$ be a corresponding automorphism of $V_\Lambda$ that is fixed-point free on $\Lambda$.  We define $L_{g^*} = OCQ(AT^i_{I\!I_{1,1}(-1/N)}({}^{g^*}_NV_\Lambda))$ for some orbifold-admissible quadratic isomorphism $i: (I\!I_{1,1}(-1/N)/I\!I_{1,1}(-N), e(Q)) \to (\bZ/N\bZ \times \bZ/n\bZ, (a,b) \mapsto e(ab/N))$, and we let $F^{g^*}$ be the vector-valued modular function defined by $F^{g^*}_{m,n}(\tau) = F^g_{n,m}(\tau)$, with coefficients $c^{g^*}_{m,n}(k) = c^g_{n,m}(k)$.
\end{defn}

\begin{prop} \label{prop:conway-lie-algebra-is-monstrous-lie-algebra}
Let $g$ be a non-Fricke element in $\bM$, and let $g^*$ be a corresponding automorphism of $V_\Lambda$ that is fixed-point free on $\Lambda$, and satisfies the ``no massless states'' condition.  Then switching the coordinates in the grading group $\bZ \times \bZ$ induces a $\widetilde{C_{\bM}(g)}$-equivariant Lie algebra isomorphism $\fm_g \cong L_{g^*}$.
\end{prop}
\begin{proof}
By Corollary \ref{cor:orbifold-isomorphism-for-abelian-intertwining-algebras}, switching coordinates in the grading group induces an abelian intertwining algebra isomorphism ${}^g_NV^\natural \cong {}^{g^*}_NV_\Lambda$, up to coboundary adjustment.  Then applying the ``add a torus and quantize'' functor $OCQ \circ AT_{I\!I_{1,1}(-1/N)}$ to this isomorphism, where we identify $I\!I_{1,1}(-1/N)$ with $\bZ \times \bZ$ with quadratic form $(a,b) \mapsto -\frac{ab}{N}$, yields an isomorphism of Lie algebras.
\end{proof}

\begin{cor}
The Lie algebra $L_{g^*}$ is a rank 2 Borcherds-Kac-Moody Lie algebra, graded by $\bZ \times \bZ$, with Weyl vector $(0,-1)$ and its denominator formula is given by the identity
\[ T_g(\tau) - T_g(-1/\sigma) = q^{-1} \prod_{r \in \bZ_{>0}, s \in \frac{1}{N}\bZ} (1-q^r p^s)^{c^{g^*}_{r,s}(rs)} \]
In particular, the Weyl vector is $(0,1/N)$.
\end{cor}
\begin{proof}
The Borcherds-Kac-Moody property is invariant under isomorphism, so the first claim follows from Proposition \ref{prop:conway-lie-algebra-is-monstrous-lie-algebra}.  The displayed identity is precisely the denominator formula for the Lie algebra $L_g$, as given in Theorem 4.2 of \cite{GM2}, but with $p$ and $q$ switched.  Similarly, the coordinates of the Weyl vector are suitably adjusted.
\end{proof}

We then obtain twisted denominator identities that are identical to those in \cite{GM2}, but with variables switched.

\begin{prop} \label{prop:conway-denominator}
Let $h \in \widetilde{C_{\bM}(g)}$, with its action on $L_{g^*}$ induced by its action on ${}^g_NV^\natural$.  Let $V^{i,j/N}_k$ denote the subspace of $({}^g_NV^\natural)_{j,i}$ on which $L(0)$ acts by $k$, or equivalently, the subspace of $({}^{g^*}_NV_\Lambda)_{i,j}$ on which $L_0$ acts by $k$.  Then for any $h \in \widetilde{C_{\bM}(g)}$ (or equivalently, $\widetilde{C_{\Aut V_\Lambda}(g^*)}$) we have the following twisted denominator formula:

\[ \begin{aligned}
q^{-1} &\exp \left( - \sum_{i>0} \sum_{s \in \frac{1}{N} \bZ_{> 0}} \mathrm{Tr}(h^i|V^{0,s}_1) q^{is}/i \right) - \sum_{r \in \bZ_{\geq 0}} \mathrm{Tr}(h|V^{r,1/N}_{1+r/N}) p^r = \\
& = q^{-1} \exp \left( - \sum_{i > 0} \sum_{r \in \bZ_{\geq 0}, s \in \frac{1}{N} \bZ_{> 0}} \mathrm{Tr}(h^i | V^{r,s}_{1+sr})q^{is}p^{ir}/i \right) .
\end{aligned} \]
\end{prop}
\begin{proof}
By the isomorphism in Proposition \ref{prop:conway-lie-algebra-is-monstrous-lie-algebra}, we just need the twisted denominator formula given in Proposition 4.7 of \cite{GM2}, but with the variables $p$ and $q^{1/N}$ switched.
\end{proof}

\begin{rem}
When $\sigma$ is a fixed-point free automorphism of $\Lambda$, but doesn't satisfy the ``no masssless states'' condition, we can still make a rank 2 Lie algebra $L_\sigma$, but we find that there are norm zero simple roots on both axes.  By the same argument as in Theorem \ref{thm:mg-is-bkm}, the Borcherds-Kac-Moody condition for this Lie algebra is then equivalent to the claim that $V_\Lambda/\sigma \cong V_\Lambda$.  Sven M\"oller has recently informed me that he has proved this claim for all 42 conjugacy classes (which form 39 algebraic conjugacy classes) of such $\sigma$.  It is therefore natural to consider Moonshine questions for these cases as well.
\end{rem}

\begin{defn}
Let $g$ and $h$ be commuting automorphisms of $V_\Lambda$.  We define the orbifold functions $Z_\Lambda(g,h;\tau)$ (up to a root of unity ambiguity) by choosing a lift $\tilde{h}$ of $h$ to a finite order automorphism of $V_\Lambda(g)$, and setting $Z_\Lambda(g,h;\tau) = \Tr(\tilde{h}q^{L_0-1}|V_\Lambda(g))$.
\end{defn}

\begin{prop}
For any $\left(\begin{smallmatrix} a & b \\ c & d \end{smallmatrix} \right) \in SL_2(\bZ)$ and any commuting pair $(g,h)$ in $\Aut V_\Lambda$, $Z(g,h,\frac{a\tau+b}{c\tau+d})$ is proportional to $Z(g^a h^c, g^b h^d,\tau)$.
\end{prop}
\begin{proof}
This follows immediately from Theorem 6.2 in \cite{CM16}, since $V_\Lambda$ is a simple holomorphic $C_2$-cofinite vertex operator algebra.
\end{proof}

We then find ourselves with behavior that is quite similar to what we found in Generalized Monstrous Moonshine, so we have the following natural questions:
\begin{enumerate}
\item Does the twisted denominator formula in Proposition \ref{prop:conway-denominator} imply $Z_\Lambda(g,h;\tau)$ is Hecke-monic when $g$ is fixed-point free and satisfies the ``no massless states'' condition?
\item If $g$ is fixed-point free and satisfies the ``no massless states'' condition, are the functions $Z_\Lambda(g,h;\tau)$ either constants or Hauptmoduln?
\item If $g$ does not satisfy the two conditions, what can we say about $Z_\Lambda(g,h;\tau)$?
\end{enumerate}

In 2009, Tuite suggested to me that questions about $V_\Lambda$-moonshine and the fine structure of non-Fricke twists of $V^\natural$ are susceptible to exhaustive computational attack, since in principle everything can be written explicitly in terms of the Leech lattice.  These new results seem to add weight to his claim, but I do not seem to have the computational fortitude to make serious progress.

\section{Appendix} \label{sect:appendix}

%To get GAP to display the characters of the 24-dim irrep of $Co_0$ with powermaps, use the following:
%T := CharacterTable("2.Co1");;
%options := rec(chars:= 102, powermap:=true);;
%Display(T, options);

%To produce characters from root multiplicities, use M\"obius function.

We produce a table giving the correspondence between non-Fricke non-anomalous classes in $\bM$ and fixed-point free non-anomalous classes in $Co_0$ in \cite{GAP4} and ATLAS notation.
The subscripts in the ATLAS notation come from the separation of classes in $Co_1$, and are explained in \cite{ATLAS} section 7.8, page xxvi.  As part of the computation, we include non-negative eta expansions (see Definition \ref{defn:non-negative}), which are reciprocals of the corresponding Frame shapes of $Co_0$ classes, together with norm zero root multiplicities in the Lie algebra $\fm_g$.  The multiplicity at $(k,0)$ for an element of order $n$, given by $\sum_{d|(k,n)} a_d$, is only listed for $k|n$ in increasing order, since the multiplicities only depend on $\gcd(k,n)$.  We add the exponent $\phi(n/k)$ to indicate how many times the root appears modulo $n$.
\newpage

\begin{tabular}{l|l|l|l|l}
$\bM$ & eta product & norm zero roots & GAP & ATLAS \\ \hline 
2B & $1^{24}/2^{24}$ & $24^1, 0^1$ & 2a & $1A_1$ \\
3B & $1^{12}/3^{12}$ & $12^2, 0^1$ & 3a & $3A_0$ \\
4C &  $1^8/4^8$ & $8^2, 8^1, 0^1$ & 4c & $4A_1$ \\
5B & $1^6/5^6$ & $6^4, 0^1$ & 5a & $5A_0$ \\
6C & $1^6 3^6/2^6 6^6$ & $6^2, 0^2, 12^1, 0^1$ & 6b & $3B_1$ \\ 
6D & $1^4 2^4/3^4 6^4$ & $4^2, 8^2, 0^1, 0^1$ & 6e & $6A_0$ \\
6E & $1^5 3^1/2^1 6^5$ & $5^2, 4^2, 6^1, 0^1$ & 6j & $6D_1$ \\
7B & $1^4/7^4$ & $4^6, 0^1$ & 7a & $7A_0$ \\
8E & $1^4 4^2 /2^2 8^4$ & $4^4, 2^2, 4^1, 0^1$ & 8e & $8C_1$ \\
9B & $1^3/9^3$ & $3^6, 3^2, 0^1$ & 9a & $9A_0$ \\
10B & $1^4 5^4/2^4 10^4$ & $4^4, 0^4, 8^1, 0^1$ & 10b & $5B_1$ \\
10C & $1^2 2^2/ 5^2 10^2$ & $2^4, 4^4, 0^1, 0^1$ & 10d & $10A_0$ \\
10E & $1^3 5^1/ 2^1 10^3$ & $3^4, 2^4, 4^1, 0^1$ & 10i & $10E_1$ \\
12B & $1^4 4^4 6^4/2^4 3^4 12^4$ & $4^4, 0^2, 0^2, 4^2, 0^1, 0^1$ & 12c & $12A_0$ \\
12E & $1^2 3^2/4^2 12^2$ & $2^4, 2^2, 4^2, 0^2, 4^1, 0^1$ & 12j & $12E_1$ \\
12I & $1^3 4^1 6^2/2^2 3^1 12^3$ & $3^4, 1^2, 2^2, 2^2, 2^1, 0^1$ & 12r & $12K_1$ \\
13B & $1^2/13^2$ & $2^{12}, 0^1$ & 13a & $13A_0$ \\
14B & $1^3 7^3/2^3 14^3$ & $3^6, 0^6, 6^1, 0^1$ & 14b & $7B_1$ \\
15B & $1^2 5^2/3^2 15^2$ & $2^8, 0^4, 4^2, 0^1$ & 15b & $15B_0$ \\
16B & $1^2 8^1/2^1 16^2$ & $2^8, 1^4, 1^2, 2^1, 0^1$ & 16c & $16B_1$ \\
18A & $1^1 2^1/9^1 18^1$ & $1^6, 2^6, 1^2, 2^2, 0^1, 0^1$ & 18d & $18A_0$ \\
18C & $1^3 6^2 9^3/2^3 3^2 18^3$ & $3^6, 0^6, 1^2, 0^2, 4^1, 0^1$ & 18c & $9C_1$ \\
18D & $1^2 6^1 9^1/2^1 3^1 18^2$ & $2^6, 1^6, 1^2, 1^2, 2^1, 0^1$ & 18g & $18B_1$ \\
20C & $1^2 4^2 10^2/2^2 5^2 20^2$ & $2^8, 0^4, 2^4, 0^2, 0^1, 0^1$ & 20c & $20A_0$ \\
21B & $1^1 3^1/7^1 21^1$ & $1^{12}, 2^6, 0^2, 0^1$ & 21b & $21B_0$ \\
22B & $1^2 11^2/2^2 22^2$ & $2^{10}, 0^{10}, 4^1, 0^1$ & 22a & $11A_1$ \\
24C & $1^2 6^1 8^2 12^1/2^1 3^2 4^1 24^2$ & $2^8, 1^4, 0^4, 0^2, 0^2, 2^2, 0^1, 0^1$ & 24d & $24B_0$ \\
28C & $1^1 7^1/4^1 28^1$ & $1^{12}, 1^6, 0^6, 2^2, 2^1, 0^1$ & 28c & $28A_1$ \\
30A & $1^3 6^3 10^3 15^3/2^3 3^3 5^3 30^3$ & $3^8, 0^8, 0^4, 0^2, 0^4, 0^2, 0^1, 0^1$ & 30a & $15A_1$ \\
30C & $1^1 3^1 5^1 15^1/2^1 6^1 10^1 30^1$ & $1^8, 0^8, 2^4, 2^2, 0^4, 0^2, 4^1, 0^1$ & 30d & $15D_1$ \\
30G & $1^2 6^1 10^1 15^2/2^2 3^1 5^1 30^2$ & $2^8, 0^8, 1^4, 1^2, 0^4, 0^2, 2^1, 0^1$ & 30e & $15E_1$ \\
33A & $1^1 11^1/3^1 33^1$ & $1^{20}, 0^{10}, 2^2, 0^1$ & 33a & $33A_0$ \\
36B & $1^1 4^1 18^1/2^1 9^1 36^1$ & $1^{12}, 0^6, 1^4, 1^6, 0^2, 0^2, 1^2, 0^1, 0^1$ & 36a & $36A_0$ \\
42B & $1^2 6^2 14^2 21^2/2^2 3^2 7^2 42^2$ & $2^{12}, 0^{12}, 0^6, 0^6, 0^2, 0^2, 0^1, 0^1$ & 42a & $21A_1$ \\
46AB & $1^1 23^1/2^1 46^1$ & $1^{22}, 0^{22}, 2^1, 0^1$ & 46ab & $23A_1B_1$ \\
60D & $1^1 12^1 15^1 20^1/3^1 4^1 5^1 60^1$ & $1^{16}, 1^8, 0^8, 0^8, 0^4, 0^4, 0^2, 0^4, 0^2, 0^2, 0^1, 0^1$ & 60d & $60A_1$ \\
70B & $1^1 10^1 14^1 35^1/2^1 5^1 7^1 70^1$ & $1^{24}, 0^{24}, 0^6, 0^4, 0^6, 0^4, 0^1, 0^1$ & 70a & $35A_1$ \\
78BC & $1^1 6^1 26^1 39^1/2^1 3^1 13^1 78^1$ & $1^{24}, 0^{24}, 0^{12}, 0^{12}, 0^2, 0^2, 0^1, 0^1$ & 78ab & $39A_1B_1$
\end{tabular}


\begin{thebibliography}{van Ekeren-M\"oller-Scheithauer 2015}

\bibitem[Abe-Lam-Yamada 2017]{ALY17}
T. Abe, C. Lam, H. Yamada,
\textit{A Remark on $\bZ_p$-orbifold constructions of the moonshine vertex operator algebra}
Math. Z. (2018). https://doi.org/10.1007/s00209-017-2036-3
Available as \url{https://arxiv.org/abs/1705.09022}

\bibitem[Borcherds 1986]{B86}
R. Borcherds,
\textit{Vertex algebras, Kac-Moody algebras, and the Monster}
Proc. Nat. Acad. Sci. USA, \textbf{83} no. 10 (1986) 3068--3071.

\bibitem[Borcherds 1988]{B88}
R. Borcherds,
\textit{Generalized Kac-Moody algebras}
J. Algebra, \textbf{115} no. 2 (1988) 501--512.

\bibitem[Borcherds 1992]{B92} 
R. Borcherds,
\textit{Monstrous moonshine and monstrous Lie superalgebras}
Invent. Math. \textbf{109} (1992), 405--444.

\bibitem[Borcherds 1995]{B95}
\textit{A characterization of generalized Kac-Moody algebras}
J. Algebra, \textbf{174} no. 3 (1995) 1073-1079.

\bibitem[Carnahan 2010]{GM1}
S. Carnahan,
\textit{Generalized Moonshine I: Genus zero functions}
Algebra and Number Theory \textbf{4} no. 6  (2010) 649--679.  Available as \url{https://arxiv.org/abs/0812.3440}

\bibitem[Carnahan 2012]{GM2}
S. Carnahan,
\textit{Generalized Moonshine II: Borcherds products}
Duke Math J. \textbf{161} no. 5 (2012) 893--950.  Available as \url{https://arxiv.org/abs/0908.4223}

\bibitem[Carnahan 2017]{Fricke}
S. Carnahan,
\textit{Fricke Lie algebras and the genus zero property in Moonshine}
J. Phys. A, To appear.  Available as \url{https://arxiv.org/abs/1701.07846}

\bibitem[Carnahan $\geq$2017]{GM4}
S. Carnahan,
\textit{Generalized Moonshine IV: Monstrous Lie algebras}
ArXiv preprint: \url{https://arxiv.org/abs/1208.6254}

\bibitem[Carnahan-Miyamoto $\geq$2017]{CM16}
S. Carnahan, M. Miyamoto,
\textit{Regularity of fixed-point vertex operator subalgebras}
ArXiv preprint \url{https://arxiv.org/abs/1603.05645}

\bibitem[Chen-Lam-Shimakura 2016]{CLS16}
H.Y. Chen, C.H. Lam and H. Shimakura
\textit{Z3-orbifold construction of the Moonshine vertex operator algebra and some maximal 3-local subgroups of the Monster}
Math. Z. (2017). https://doi.org/10.1007/s00209-017-1878-z
Available as \url{https://arxiv.org/abs/1606.05961}

\bibitem[Conway-Norton 1979]{CN79} 
J. Conway, S. Norton,
\textit{Monstrous Moonshine}
Bull. Lond. Math. Soc. \textbf{11} (1979) 308--339.

\bibitem[Conway, et al, 1985]{ATLAS}
J. Conway, R. Curtis, S. Norton, R. Parker, R. Wilson,
\textit{Atlas of finite groups}
Clarendon Press, Oxford, 1985.

\bibitem[Dong 1994]{D94}
C. Dong,
\textit{Representations of the moonshine module vertex operator algebra} 
Contemporary Math. \textbf{175} (1994) 27--36.

\bibitem[Dong-Lepowsky 1993]{DL93}
C. Dong, J. Lepowsky,
\textit{Generalized vertex algebras and relative vertex operators}
Progress in Mathematics \textbf{112} Birkh\"auser Boston, Inc., Boston, MA, (1993).

\bibitem[Dong-Lepowsky 1996]{DL96}
C. Dong and J. Lepowsky,
\textit{The algebraic structure of relative twisted vertex operators}
J. Pure, Appl. Math. \textbf{110} (1996), 259--295.  Available as \url{https://arxiv.org/abs/q-alg/9604022}.

\bibitem[Dong-Li-Mason 1997]{DLM97} 
C. Dong, H. Li, G. Mason,
\textit{Modular invariance of trace functions in orbifold theory}
Comm. Math. Phys. \textbf{214} (2000) no. 1, 1--56.  Available as \url{https://arxiv.org/abs/q-alg/9703016}.

\bibitem[Dong-Mason 1994]{DM94}
C. Dong and G. Mason,
\textit{On the construction of the moonshine module as a Zp-orbifold}
Mathematical aspects of conformal and topological field theories and quantum groups (South Hadley, MA, 1992),
37--52, Contemp. Math., \textit{175} Amer. Math. Soc., Providence, RI, 1994.

\bibitem[Dong-Mason 2002a]{DM02a}
C. Dong, G. Mason,
\textit{Rational vertex operator algebras and the effective central charge}
Int. Math. Res. Not. \textbf{56} (2004) 2989--3008.  Available as \url{https://arxiv.org/abs/math/0201318}.

\bibitem[Dong-Mason 2002b]{DM02b}
C. Dong, G. Mason,
\textit{Holomorphic Vertex operator algebras of small central charge}
Pac. J. Math. \textbf{213} no.2 (2004) 253--266.  Available as \url{https://arxiv.org/abs/math/0203005}.

\bibitem[Dong-Nagatomo 1998]{DN98}
C. Dong and K. Nagatomo,
\textit{Automorphism groups and twisted modules for lattice vertex operator algebras}
Contemp. Math. \textbf{248} (1999), 117--133.  Available as \url{https://arxiv.org/abs/math/9808088}.

\bibitem[van Ekeren-M\"oller-Scheithauer 2015]{vEMS}
J. van Ekeren, S. M\"oller, N. Scheithauer,
\textit{Construction and classification of vertex operator algebras}
ArXiv preprint \url{https://arxiv.org/abs/1507.08142}

\bibitem[Frame 1972]{F72}
J. Frame,
\textit{Computation of the characters of the Higman-Sims group and its automorphism group}
J. Algebra \textbf{20} no.2 (1972) 320--349.

\bibitem[Frenkel-Lepowsky-Meurman 1988]{FLM88} 
I. Frenkel, J. Lepowsky, A. Meurman,
\textit{Vertex operator algebras and the Monster}
Pure and Applied Mathematics \textbf{134} Academic Press, Inc., Boston, MA, (1988).

\bibitem[Gaberdiel-Persson-Volpato 2013]{GPV13}
M. Gaberdiel, D. Persson, R. Volpato,
\textit{Generalised Moonshine and Holomorphic Orbifolds}
Proceedings, String-Math 2012, Bonn, Germany, July 16--21, 2012, (2013) 73--86. Available at \url{https://arxiv.org/abs/1302.5425}

\bibitem[Ganter 2007]{G07}
\textit{Hecke operators in equivariant elliptic cohomology and generalized Moonshine}
Harnad, John (ed.) et al., Groups and symmetries. From Neolithic Scots to John McKay. AMS CRM Proceedings and Lecture Notes \textbf{47} (2009) 173--209.   Available at \url{https://arxiv.org/abs/0706.2898}

\bibitem[GAP]{GAP4}
The GAP~Group,
\textit{GAP -- Groups, Algorithms, and Programming, Version 4.5.5}
(2012) \url{http://www.gap-system.org}
 
 \bibitem[Goddard-Thorn-1972]{GT72}
P. Goddard, C. Thorn,
\textit{Compatibility of the dual pomeron with unitarity and the absence of ghosts in the dual resonance model}
Physics Letters \textbf{40B} no. 2, (1972) 235--238.

\bibitem[Griess-Lam 2011]{GL11}
R. Griess, C. Lam,
\textit{A new existence proof of the Monster by VOA theory}
Michigan Math. J. \textbf{61} no. 3 (2012), 555--573.  Available as \url{https://arxiv.org/abs/1103.1414}.

\bibitem[Johnson-Freyd 2017]{JF17}
T. Johnson-Freyd,
\textit{The Moonshine Anomaly}
ArXiv preprint \url{https://arxiv.org/abs/1707.08388}

\bibitem[Jurisich 1998]{J98}
E. Jurisich,
\textit{Generalized Kac-Moody Lie algebras, free Lie algebras and the structure of the Monster Lie algebra}
J. Pure and Appl. Alg. \textbf{126} (1998) 233--266.
Available at \url{https://arxiv.org/abs/1311.3258}

\bibitem[Jurisich 2004]{J04}
E. Jurisich,
\textit{A resolution for standard modules of Borcherds Lie algebras}
J. Pure Appl. Alg. \textbf{192} (2004) 149--158.

\bibitem[Kondo 1985]{K85}
T. Kondo,
\textit{The automorphism group of the Leech lattice and elliptic modular functions}
J. Math. Soc. Japan \textbf{37} no. 2 (1985) 337--362.

\bibitem[Lam-Shimakura 2017]{LS17}
C. Lam, H. Shimakura,
\textit{On orbifold constructions associated with the Leech lattice vertex operator algebra}
ArXiv preprint \url{https://arxiv.org/abs/1705.01281}

\bibitem[Lepowsky 1985]{L85}
J. Lepowsky,
\textit{Calculus of twisted vertex operators}
Proc. Natl. Acad. Sci. USA \textbf{82} (1985), 8295--8299.

\bibitem[Miyamoto 1997]{M97}
M. Miyamoto,
\textit{A New Construction of the Moonshine Vertex Operator Algebra over the Real Number Field}
Ann. Math. (2) \textbf{159} No. 2 (Mar., 2004) 535--596.  Available as \url{https://arxiv.org/abs/q-alg/9701012}.

\bibitem[Miyamoto 2013]{M13}
\textit{$C_2$-cofiniteness of cyclic-orbifold models}
Comm. Math. Phys. \textbf{335} (2015), no. 3, 1279--1286.
Available at \url{https://arxiv.org/abs/1306.5031}

\bibitem[Montague 1995]{M95}
P. Montague,
\textit{Third and Higher Order NFPA Twisted Constructions of Conformal Field Theories from Lattices}
Nuclear Phys. B \textbf{441} (1995), no. 1-2, 337--382.  Available as \url{https://arxiv.org/abs/hep-th/9502138}.

\bibitem[Norton 1987]{N87}
S. Norton,
\textit{Generalized moonshine}
Proc. Sympos. Pure Math. \textbf{47} Part 1,  The Arcata Conference on Representations of Finite Groups (Arcata, Calif., 1986),  209--210, Amer. Math. Soc., Providence, RI (1987).

\bibitem[Norton 2001]{N01} 
S. Norton,
\textit{From Moonshine to the Monster}
Proceedings on Moonshine and related topics (Montr\'eal, QC, 1999),  163--171,
CRM Proc. Lecture Notes \textbf{30}, Amer. Math. Soc., Providence, RI (2001).

\bibitem[Paquette-Persson-Volpato 2016]{PPV16}
N. Paquette, D. Persson, R. Volpato,
\textit{Monstrous BPS algebras and the superstring origin of moonshine}
ArXiv preprint \url{https://arxiv.org/abs/1601.05412}

\bibitem[Paquette-Persson-Volpato 2017]{PPV17}
N. Paquette, D. Persson, R. Volpato,
\textit{BPS algebras, genus zero, and the heterotic Monster}
ArXiv preprint \url{https://arxiv.org/abs/1701.05169}

\bibitem[Queen 1981]{Q81}
L. Queen,
\textit{Modular Functions arising from some finite groups}
Mathematics of Computation \textbf{37} (1981) No. 156, 547--580.

\bibitem[Tuite 1992]{T92}
M. Tuite,
\textit{Monstrous Moonshine from orbifolds}
Comm. Math. Phys. \textbf{146} (1992) 277--309.

\bibitem[Tuite 1993]{T93}
M. Tuite,
\textit{On the relationship between monstrous Moonshine and the uniqueness of the Moonshine module}
Comm. Math. Phys. \textbf{166} (1995) 495--532.  Available as \url{https://arxiv.org/abs/hep-th/9305057}.

\bibitem[Tuite 2008]{T08}
M. Tuite,
\textit{Monstrous and generalized Moonshine and permutation orbifolds}
J. Lepowsky, J. McKay, and M. Tuite (eds.) Moonshine - The First Quarter Century and Beyond, CUP (2010) 378--392.  Available at \url{https://arxiv.org/abs/0811.4525}

\bibitem[Vafa 1986]{V86}
C. Vafa,
\textit{Modular Invariance And Discrete Torsion On Orbifolds}
Nucl. Phys. B \textbf{273} (1986) 592--606.

\end{thebibliography}
\end{document}